\documentclass{amsart}

\usepackage{amsmath}
\usepackage{amsfonts}
\usepackage{amstext}
\usepackage{amsbsy}
\usepackage{amsopn}
\usepackage{amsxtra}
\usepackage{upref}
\usepackage{amsthm}
\usepackage{amsmath}
\usepackage{amssymb}

\newtheorem{prop}{Proposition}[section]
\newtheorem{rem}{Remark}[section]
\newtheorem{lema}{Lemma}[section]
\newtheorem{defi}{Definition}[section]
\newtheorem{teo}{Theorem}[section]
\newtheorem{eje}{Example}[section]

\def\R{{\mathbb R}}

\def\N{{\mathbb N}}
\def\Z{{\mathbb Z}}
\def\M{{\mathcal M}}

\title[Thermodynamic formalism for the positive geodesic flow]{Thermodynamic formalism for the positive geodesic flow on the modular surface}
\date{\today}

\begin{thanks}
{The author was partially  supported by Proyecto Fondecyt 11070050.}
\end{thanks}

\author{Godofredo Iommi} \address{Facultad de Matem\'aticas,
Pontificia Universidad Cat\'olica de Chile (PUC), Avenida Vicu\~na Mackenna 4860, Santiago, Chile}
\email{giommi@mat.puc.cl}
\urladdr{http://www.mat.puc.cl/\textasciitilde giommi/}

\begin{document}

\begin{abstract}
In this note we study the thermodynamic formalism for the positive geodesic flow on the modular surface. We define the pressure and prove  the variational principle. We also establish conditions for the the pressure to be real analytic and for the potentials to have unique equilibrium states. The results in this paper were largely superceded by \cite{ij}.\end{abstract}

\maketitle

\section{Introduction}
This note is devoted to study ergodic properties of the  positive geodesic flow on the modular surface. The results presented here were  superceded by  the ones obtained by Iommi and Jordan in \cite{ij}. This flow was introduced by Svetlana Katok and  Boris Gurevich in \cite{gk} building up on previous work by Katok \cite{k}. It is an interesting flow because it has strong symbolic properties, as we will see at the end of this section. Our main purpose
is to develop a thermodynamic formalism for this flow. We define the pressure, prove a variational principle and establish sufficient conditions for a potential to have equilibrium measures (see Section \ref{termo}). This provides not only a strong set of tools to further study this flow, but it is has interest on its own right. Indeed, since the flow is defined on a non-compact set usual techniques do not hold. We overcome this major obstruction combining a symbolic representation of this flow obtained by Gurevich and Katok \cite{gk}
and a thermodynamic formalism for suspension flow over countable Markov shifts developed by Barreira and Iommi in \cite{BI}.

Let us begin with some basic definitions (for more details we refer to the book \cite{ka} and the article \cite{ku}). Denote by $\mathcal{H}=\{z \in \mathbb{C}: \text{Im }  z >0  \}$ the upper half-plane endowed with the hyperbolic metric. Geodesics in $\mathcal{H}$ are either semi-circles which meet the boundary perpendicularly or vertical straight lines.
The geodesic flow of $\mathcal{H}$, denote by $\{\overline{\psi}_t\}$, is the flow on the unit tangent bundle,
 $T^1\mathcal{H}$, of $\mathcal{H}$ which moves $\omega \in T^1\mathcal{H}$ along the geodesic it determines at unit speed.

The group of M\"obius transformations acting on $\mathcal{H}$ by orientation preserving isometries can be identified with the group $ \text{PSL}(2, \mathbb{Z})$. The \emph{modular surface} is defined by $M=\text{PSL}(2, \mathbb{Z})  \backslash  \mathcal{H}$, which is a (non-compact) surface of  constant negative curvature. Topologically, is a sphere with one cusp and two singularities. Note that $T^1\mathcal{H}$ can be identified with $\text{PSL}(2, \mathbb{R})$ by sending $\omega=(z, \zeta) \in T^1\mathcal{H} $ onto the unique $ g \in  \text{PSL}(2, \mathbb{R})$ such that $z=g(i)$ and $\zeta=g'(z)(\iota)$, where $\iota$ is the unit vector at the point $i$ to the imaginary axis pointing up. In this coordinate system the geodesic flow takes the form
\begin{equation*}
\overline{\psi}_t
\left(  \begin{array}{cc}   a & b \\ c & d \end{array} \right) = \left(  \begin{array}{cc}   a & b \\ c & d \end{array} \right) \left(  \begin{array}{cc}   e^{t/2} & 0 \\ 0 & e^{-t/2} \end{array} \right).\end{equation*}
The geodesic flow  $\{\overline{\psi}_t\}$ on $\mathcal{H}$ descends to the geodesic flow
 $\{\psi_t\}$ on $M$ via the projection $\pi:T^1\mathcal{H} \mapsto T^{1}M$. We will be interested in an invariant sub-system of  $\{\psi_t\}$. In order to define it we need to
 consider the fundamental region for $\text{PSL}(2, \mathbb{Z})$ given by
\begin{equation*}
F=\{z \in \mathcal{H} : |z| \geq 1, | \text{Re } z | \leq 1/2 \},
\end{equation*}
whose sides are identified by the generators of  $\text{PSL}(2, \mathbb{Z})$, $T(z)=z+1$ and $S(z)=-1/z$ (see \cite[p.55]{ka}). Any geodesic can be represented by a series of segments in $F$. We will only be interested in oriented geodesic that do not go to the cusps of $M$ in either direction.

\begin{defi}
A geodesic $\gamma$ is called \emph{positive} if  all segments comprising the geodesic $\gamma$  in F are positively (clockwise) oriented. The set of vectors in $T^1M$ tangent to positive geodesics is a  non-compact invariant set of the geodesic flow on $T^1M$. We call the restriction of  the geodesic flow to this set the \emph{positive geodesic flow}.
\end{defi}

Positive geodesic have  interesting coding  properties.  There are several ways to represent geodesics by symbolic sequences, for a detailed exposition on the subject see \cite{ku}. Here we will be interested in two if them:

\subsection{The geometric code}

As we just saw, the sides of $F$ are identified with the generators of $\text{PSL}(2, \mathbb{Z})$,
$T(z)= z+1$ and $S(z)= -1/z$. The \emph{geometric code} (also known as Morse code) of a geodesic in $F$ (that does not go to the cusp in either direction)
is a bi-infinite sequence of integers. The idea is to code the geodesic by recording the sides of the region $F$ that are cutted by the geodesic.
The boundary of $F$ has three sides, the left and right vertical sides are labeled by $T$ and $T^{-1}$, respectively. The circular boundary is labeled by $S$.  Any oriented geodesic, that does not go to the cusp, returns to the circular boudary of $F$ infinitely often. The geometric code is obtained as follows. Choose an initial point on the circular boundary of $F$ and count the number of times it hits the vertical boundary of $F$ moving  in the direction of the geodesic. We assign a positive integer to each block of hits of the right vertical side and a negative number to the left vertical side. if we move the initial point in the opposite direction we obtain a sequence of nonzero integers
\[ [\gamma]=[ \dots ,n_{-1}, n_0, n_1, \dots ] \]
that we denote \emph{geometric code}.

\subsection{The arithmetic code}
An oriented  geodesic $\gamma \in \mathcal{H}$ is called \emph{reduced} if its endpoints
$u,v$ satisfiy $0<u<1$ and $w>1$. Recall that  a geodesic which does not go to the cusp in either direction is such that its end points are irrationals.  Consider the minus continued fraction associated to $u$ and $v$,
\begin{equation*}
w= \textrm{ } n_1 +\cfrac{1}{n_2 - \cfrac{1}{n_3 - \cfrac{1}{n_4 - \dots}}} \text{ , }  \frac{1}{u} = \textrm{ } n_0 -\cfrac{1}{n_{-1} - \cfrac{1}{n_{-2} - \cfrac{1}{n_{-3} - \dots}}} \end{equation*}
The following code is given to $\gamma$,
\[(\gamma)= ( \dots, n_{-2}, n_{-1}, n_{0}, n_1, n_2, \dots).\]
Now, it is possible to show that an arbitrary geodesic $\gamma \in M$ can also be represented by doubly infinite sequence. The idea is to construct a cross section  for the geodesic flow on $T^1M$. Every oriented geodesic can be represented as a bi-infnite sequence of segments $\sigma_i$ between returns to the cross section. To each segment  it corresponds a reduced geodesic $\gamma_i$. It turns out that (see \cite{gk}) all these geodesics have the same arithmetic code, except for a shift. This sequence is the \emph{arithmetic code} of $\gamma$,
\[(\gamma)= (\dots, , n_{-2}, n_{-1}, n_{0}, n_1, n_2, \dots).\]

\subsection{Equality of codes}

Positive geodesic behave very well with respect to the arithmetic and geometric codes. Indeed, Gurevich and Katok \cite{gk} proved the following result which is a generalisation of earlier results by Katok \cite{k},

\begin{teo}[Gurevich-Katok] A geodesic $\gamma$ is positive if and only if the arithmetic and the geometric codes for $\gamma$ coincide.
\end{teo}

In this note we  study the thermodynamic formalism for the positive geodesic flow. These are ideas and techniques that come originally  from statistical mechanics and were brought into dynamical systems by Ruelle and Sinai in the early seventies. This formalism  provides procedures for the choice of invariant measures. Let us stress that the positive geodesic flow has many invariant measures, hence the problem of choosing relevant ones.  A good understanding of the thermodynamic formalism could allow us, for instance,  to develop a dimension theory (e.g multifractal formalism, Bowen formula)  for this flow, see \cite{pe} for details.

Thermodynamic formalism was studied for the geodesic flow on compact negatively curved manifolds (more generally for Axiom A flows) by Bowen and Ruelle \cite{br}. They proved that, in that context, the geodesic flow can be coded by a suspension flow over a sub-shift of finite type defined over a finite alphabet (this is done using the Markov partitions obtained by Bowen \cite{bo1} and Ratner \cite{ra}). This relation allowed them to reduce the study of the geodesic flow to  a suspension flow. They proved existence and uniqueness of equilibrium measures (see Section \ref{termo} for precise definitions)  for a wide range of potentials.

It was shown by Gurevich and Katok (see Section \ref{susp} and \cite{gk}) that the positive geodesic flow in the modular surface can be coded by a suspension flow over a countable Markov shift. The fact that the model is not compact is a major obstruction if ones wants to apply the results by Bowen and Ruelle. Nevertheless, thermodynamic formalism for suspension flows over countable Markov shifts was developed and studied by Barreira and Iommi in \cite{BI}. In this note, making use of the results in \cite{gk} and \cite{BI}, we provide a definition of pressure for the positive geodesic flow. Moreover, we obtain a variational principle and establish conditions under which a large class of potentials have equilibrium measures.

\section{Preliminaries from Ergodic theory}
In this section we review results form ergodic theory that will be used in the rest of the paper. First we recall the definition of pressure for countable Markov shifts and some of its properties. We then recall the suspension flow construction and study the relation between the invariant measures for the shift and the invariant measures for the suspension flow.  This relation will be essential in what follows because it allow us to relate the study of the thermodynamic formalism for the flow with the study of the thermodynamic formalism for the shift.

\subsection{Thermodynamic formalism for countable Markov shifts} \label{cms}

Let $B$ be a transition matrix defined on the alphabet of natural numbers. That is, the entries of the matrix
$B=B(i,j)_{\N \cup \left\{0\right\}  \times \N \cup \{0\} }$  are zeros and ones (with no row and no column
made entirely of zeros). The countable Markov shift $(\Sigma_B, \sigma)$
is the set
\[ \Sigma_B:= \left\{ (x_n)_{n \in \N \cup \{0\}} : B(x_n, x_{n+1})=1  \text{ for every } n \in \N  \cup \{0\} \right\}, \]
together with the shift map $\sigma: \Sigma_B \to \Sigma_B$ defined by
$\sigma(x_0, x_1,x_2,x_3 \dots)=(x_1, x_2,x_3,\dots)$.

\begin{rem}
Analogously, we can define a two-sided countable Markov shift by
\[ \Sigma_B^*:= \left\{ (x_n)_{n \in \Z} : B(x_n, x_{n+1})=1  \text{ for every } n \in \Z \right\}, \]
together with the shift map  $\sigma: \Sigma_B^* \to \Sigma_B^*$ defined by $(\sigma x)_n=x_{n+1}$.
\end{rem}

Recall that the space $\Sigma_B$ is equipped with
the topology generated by the cylinder sets
\begin{equation*}
C_{a_0 \cdots a_n}= \{ x\in \Sigma_B: x_i=a_i \ \text{for
$i=0,\ldots,n$}\}.
\end{equation*}
Given a function $ \rho \colon \Sigma_B \to \R$ we define
\[ V_{n}(\rho):= \sup \{| \rho(x)- \rho(y)| : x,y\in \Sigma_B, \ x_{i}=y_{i}
\ \text{for $i=0,\ldots,n-1$} \},
\]
where $x=(x_0 x_1 \cdots)$ and $y=(y_0y_1 \cdots)$. We say that
$\rho$ is \emph{locally H\"older} if there exist constants $K>0$ and
$\theta\in (0,1)$ such that $V_{n}( \rho) \le K \theta^{n}$ for all
$n\in \N$. Note that since nothing is required for $n=0$ a locally
H\"older function is not necessarily bounded.

There are essentially two notions of (topological) pressure for a
countable Markov shift. The first was proposed by Mauldin and Urba\'nski \cite{mu} and the second by Sarig \cite{sa1}. For the class of Markov shifts that we will be interested in this note both notions coincide.

\begin{defi} \label{presion}
Let $\rho \colon \Sigma_B \to \R$ be a locally H\"older function. The
\emph{pressure} of $\rho$  is defined by
\[
 P_\sigma(\rho) = \lim_{n \to
\infty} \frac{1}{n} \log \sum_{x:\sigma^{n}x=x} \exp \left(
\sum_{i=0}^{n-1} \rho(\sigma^{i}x)\right).
\]
 \end{defi}

\begin{rem} \label{formula}
Let $(\Sigma_F, \sigma)$ be the full-shift on countably many symbols, that is
\[ \Sigma_F:= \left\{ (x_n)_{n \in \N\cup \{0\}} : x_n \in \N \cup \{0\}\right\}.\]
If $\rho :\Sigma_F \to \R$ is a locally constant potential, that is
$\rho|C_n:= \log \lambda_n$, then there is a simple formula for the pressure (see, for example, \cite[Example 1]{BI})
\begin{equation}
P_{\sigma}(\rho )= \log \sum_{n=0}^{\infty} \lambda_n.
\end{equation}
\end{rem}
This notion of pressure satisfies the variational principle (see \cite{mu,sa1}),
\begin{teo}
Let $(\Sigma_B, \sigma)$ be a countable Markov shift and $\rho \colon \Sigma_B \to \R$ be a locally H\"older function, then
\[P_{\sigma}(\rho)= \sup \left\{ h(\nu) + \int \rho \ d \nu : \nu \in \M_{\sigma} \text{ and } - \int \rho \ d \nu < \infty \right\},\]
where $ \M_{\sigma} $ denotes the set of $\sigma-$invariant probability measures and $h(\nu)$ denotes the entropy of the measure $\nu$ (for a precise definition see  \cite[Chapter 4]{wa}).
\end{teo}
A measure $\nu \in \M_{\sigma}$ attaining the supremum, that is, $P_{\sigma}(\rho)= h(\nu) + \int \rho \ d \nu$ is called \emph{equilibrium measure} for $\rho$.

We say that $\mu \in  \M_\sigma$ is a \emph{Gibbs measure}  for the
function $\rho \colon \Sigma_B \to \R$ if for some constants $P$,
$C>0$ and every $n\in \N$ and $x\in C_{a_0 \cdots a_n}$ we have
\[
\frac{1}C \le \frac{\mu(C_{a_0\cdots a_n})}{\exp (-nP + \sum_{i=0}^n
\phi(\sigma^k x))} \le C.
\]

There is a class of countable Markov shifts, introduced by Sarig \cite{sa2}, that have simple combinatorics,  similar to that of the full-shift.  The thermodynamic formalism for Markov shifts belonging to this class is similar to the one of sub-shift of finite type defined on finite alphabets. We say that a countable Markov shift $(\Sigma_B, \sigma)$, defined by the transition matrix $B(i,j)$ with $(i,j)\in \N \cup\{0\} \times \N \cup\{0\}  $, satisfies the \emph{BIP property} if and only if  there exists $\{b_1 , \dots , b_n\}  \in \N \cup\{0\} $ such that for every $a \in \N \cup\{0\} $ there exists $i,j \in \N$ with $B(b_i, a)B(a,b_j)=1$.
The following theorem sumarises results  proven by Sarig in \cite{sa2, sa3} and by Mauldin and Urba\'nski \cite{mu},

\begin{teo}
Let $(\Sigma_B, \sigma)$ be a countable Markov shift satisfying the BIP property and $\rho: \Sigma_B \to \R$ a locally H\"older potential. Then, there exists $t^* >0$ such that pressure function $t \to P(t\rho)$ has the following properties
\begin{equation*}
P_{\sigma}(t \rho)=
\begin{cases}
\infty  & \text{ if  } t \leq t^* \\
\text{real analytic } & \text{ if  } t > t^*.
\end{cases}
\end{equation*}
Moreover, if $t> t^*$, there exists a unique equilibrium measure for $t \rho$.
If $\sum_{n=1}^{\infty} V_{n}( \rho) < \infty$ and $P_{\sigma}(\rho)< \infty$ then there exists a Gibbs measure for $\rho$.
\end{teo}

\subsection{Suspension flows and invariant measures}

Let $(\Sigma_B, \sigma)$ be a countable Markov shift and $\tau \colon \Sigma_B \to \R^+$ be a positive continuous function. Consider the space
\begin{equation}\label{shift}
Y= \{ (x,t)\in \Sigma_B \times \R \colon 0 \le t \le\tau(x)\},
\end{equation}
with the points $(x,\tau(x))$ and $(\sigma(x),0)$ identified for
each $x\in \Sigma_B$. The \emph{suspension semiflow} over $\sigma$
with \emph{height function} $\tau$ is the semiflow $\Phi = (
\varphi_t)_{t \ge 0}$ on $Y$ defined by
\[
 \varphi_t(x,s)= (x,
s+t) \ \text{whenever $s+t\in[0,\tau(x)]$.}
\]
In the case of two-sided Markov shifts we can define a suspension
flow $(\varphi_t)_{t\in \R}$ in a similar manner.

We denote by $\M_\Phi$ the space of $\Phi$-invariant probability
measures on $Y$. Recall that a measure $\mu$ on $Y$ is
\emph{$\Phi$-invariant} if $\mu(\varphi_t^{-1}A)= \mu(A)$ for every
$t \ge 0$ and every measurable set $A \subset Y$. We also consider
the space $\M_\sigma$ of $\sigma$-invariant probability measures on
$\Sigma_B$.  There is a strong relation between this two spaces of invariant measures.
Consider the space of $\sigma-$invariant measures for which $\tau$ is integrable,\begin{equation}
\M_\sigma(\tau):= \left\{ \mu \in \mathcal{M}_{\sigma}: \int \tau \ d \mu < \infty \right\}.
\end{equation}
Denote by $m$ the one dimensional Lebesgue measure and let $\mu \in \M_\sigma(\tau)$ then  it follows directly from classical results by Ambrose and Kakutani \cite{ak} that
\[(\mu \times m)|_{Y} /(\mu \times m)(Y) \in \M_{\Phi}.\]
The  behaviour of the map  $R \colon \M_\sigma \to \M_\Phi$, defined by
\begin{equation} \label{R}
R(\mu)=(\mu \times m)|_{Y} /(\mu \times m)(Y)
\end{equation}
 is closely related to the ergodic properties of the flow. Indeed, in the compact setting the map $R \colon \M_\sigma \to \M_\Phi$ is a bijection. This fact was used by Bowen and Ruelle \cite{br} to study and develop the thermodynamic formalism for Axiom A flows (these flows admit a compact symbolic representation). In
the general (non-compact) setting there are several difficulties that can arise. For instance,  the height function $\tau$ need not to be bounded above. It is, therefore, possible for a measure $\nu \in \M_\sigma$ to be such that $\int \tau \ d \nu  = \infty$. In this situation the measure $\nu \times m$ is an infinite sigma-invariant measure for $\Phi$. Hence, the map $R(\cdot)$ is not well defined and this makes it harder to reduce the study of the thermodynamic formalism of the flow $\Phi$ to the one of the shift $\sigma$. Another possible complication occurs if the height function $\tau$ is not bounded away from zero. Then it is possible that for an infinite (sigma-finite) invariant measure $\nu$  we have $\int  \tau\,d\nu <\infty$. In this case the measure $(\nu \times m)|_{Y} /(\nu \times m)(Y) \in \M_\Phi$. In such a situation,   the map $R$ is not surjective. Again, the fact that $R$ is not a bijective map makes it hard to translate problems from the flow to the shift.

The following can be obtained  directly from the results by Ambrose and Kakutani \cite{ak},
\begin{lema}
 If $\tau: \Sigma_B \to \R$ is bounded away from zero, the map $R \colon
\M_\sigma(\tau) \to \M_\Phi$  defined by
\begin{equation*}
R(\mu)=(\mu \times m)|_{Y} /(\mu \times m)(Y)
\end{equation*}
  is  bijective.
\end{lema}
Given a continuous function $F \colon Y\to\R$ we define the function
$\Delta_F\colon\Sigma\to\R$~by
\[
\Delta_F(x)=\int_{0}^{\tau(x)} F(x,t) \, dt.
\]
The function $\Delta_F$ is also continuous, moreover
\begin{equation} \label{rela}
\int_{Y} F \, dR(\nu)= \frac{\int_\Sigma \Delta_F\, d
\nu}{\int_\Sigma\tau \, d \nu}.
\end{equation}

\begin{rem}[Extension of potentials defined on the base] \label{exten}
Let $\rho \colon \Sigma_B \to \R$ be a locally H\"older potential. It is shown in
\cite{brw} that there exists a continuous function $F \colon Y \to
\R$ such that $\Delta_F=\rho$. This provides a tool to construct
examples.
\end{rem}

\subsection{Abramov's formula} \label{abramov}
In this short subsection we recall a classical result by Abramov.  The entropy of a flow with respect to an invariant measure can be defined by  the entropy of the corresponding time one map. For the definition of entropy in the context of maps see \cite[Chapter 4]{wa}. In 1959 Abramov \cite{a} showed that  in the context of suspension semiflows (among others) the following relation  holds,

\begin{prop}[Abramov]
Let $\mu \in \M_{\Phi}$ be such that  $\mu=(\nu \times m)|_{Y} /(\nu \times m)(Y)$, where $\nu \in \M_{\sigma}$ then
\begin{equation}
h_{\Phi}(\mu)=\frac{h_{\sigma}(\nu)}{\int \tau \ d \nu}.
\end{equation}
\end{prop}

It follows from Abramov's result that

\begin{lema}
Let $\mu \in \M_{\Phi}$ be such that $\mu=(\nu \times m)|_{Y} /(\nu \times m)(Y)$, we have that  $h_{\Phi}(\mu)= \infty$ if and only if
$h_{\sigma}(\nu)= \infty$.
\end{lema}

When the phase space is non-compact  there are several different notions of topological entropy of a flow, we will consider the following,
\begin{defi}
The topological entropy of the suspension flow $(Y ,\Phi)$ denoted by $h(\Phi)$ is defined by
\begin{equation*}
h(\Phi)= \sup \left\{  h_{\Phi}(\mu) : \mu \in \M_{\Phi}   \right\}.
\end{equation*}
\end{defi}

\section{Suspension flow representation for the positive geodesic flow} \label{susp}
It was shown by Gurevich and Katok in \cite{gk} that the positive geodesic flow on the modular surface has a symbolic representation  as a suspension flow. In this section we recall this construction and describe some of its properties.

Consider the alphabet $\mathcal{A}=\{3,4,5,6,\dots \}$.  Let $A$ be the transition matrix defined by $A(i,j)=1$ for every pair $(i,j) \in \mathcal{A} \times \mathcal{A}$ except for the pairs $(i,j) \in \{(3,3) , (3,4), (3,5), (4,3), (5,3) \}$, where $A(i,j)=0$. Note that the pairs for which the matrix has an entry equal to zero correspond to the five Platonic bodies. Denote by $(\Sigma_A, \sigma)$ the corresponding (non-compact) countable Markov shift. The proof of the following result is straight forward.
\begin{lema}
The countable Markov shift $(\Sigma_A, \sigma)$ satisfies the BIP property.
\end{lema}

\begin{rem} \label{sub-full}
Note that the full-shift on the alphabet $\mathcal{A}'=\{6,7,8,\dots\}$, that we will denote by
$(\Sigma_F,\sigma)$, is a sub-shift of $(\Sigma_A,\sigma)$.
\end{rem}


Constructing an appropriate cross section and computing the first return time function of the geodesic flow to it, Gurevich an Katok \cite{gk} showed that the positive geodesic flow on the modular surface can be represented by the suspension flow  $(Y^*, \Phi^*)$. This flow is defined over the countable (two-sided) Markov shift $(\Sigma_A^*, \sigma)$, where
\[\Sigma_A^*:=\left\{ (x_n)_{n \in \Z} : A(x_n, x_{n+1})=1  \text{ for every } n \in \Z \right\},\]
and it has height function $\tau^*(x)= 2 \log w(x)$, where
\[w(x)= \textrm{ } n_1 +\cfrac{1}{n_2 - \cfrac{1}{n_3 - \cfrac{1}{n_4 - \dots}}}, \]
for $x=(\dots n_{-1}, n_0, n_1, n_2 , n_3 \dots)$. Note that the height function only depends on the \emph{future} coordinates of $x \in \Sigma_A^*$.

 The result in \cite{gk} is that  there exits a continuous bijection, $\pi: Y^* \to T^1 M^+$, where $ T^1 M^+$ is the space of  unit vectors tangent to positive geodesics. With the property   that for every $t \in \R$ we have $\pi \circ \phi_t (x) = \psi_t \circ \pi (x)$. Sumarising,

\begin{teo}[Gurevich-Katok]
The suspension flow $(Y^*, \Phi^*)$  is a symbolic representation of the positive geodesic flow on the modular surface.
\end{teo}

Since every potential $\overline{\rho}:\Sigma_A^* \to \R$ is cohomologous to a potential
 $\rho:\Sigma_A^* \to \R$ that only depends on \emph{future} coordinates, we can reduce the study of the suspension flow to that of the suspension semi-flow. This is a standard procedure, see for example \cite[p.93]{pp}. Therefore, in order to study the ergodic theory of the positive geodesic flow it is enough to understand the ergodic theory of the suspension semi-flow $(Y,\Phi)$ defined over the countable Markov shift $(\Sigma_A, \sigma)$ with height function $\tau:\Sigma_A \to \R$ defined by $\tau(x)= 2\log w(x)$.

\begin{rem}
Let $c= (3 + \sqrt{5})/6$. If $x=(\dots n_{-1}, n_0, n_1, n_2 , n_3 \dots)$ then
\begin{equation*}
2 \log (c n_1) \leq \tau(x) \leq 2 \log n_1.
\end{equation*}
\end{rem}

\section{Thermodynamic formalism for the positive geodesic flow} \label{termo}
The main purpose of this note is to define the pressure function for the positive geodesic flow on the modular surface and to study its properties. In order to do so, we will use the symbolic representation explained in the previous section together with results of Barreira and Iommi \cite{BI} on thermodynamic formalism for suspension flows over countable Markov shifts.

As we have seen in Section  \ref{susp} it is enough to study the suspension semi-flow $(Y, \Phi)$. We start by defining the class of potentials that we will consider. Denote by
\begin{equation}
\mathcal{P}:= \left\{F:Y \to \R: \text{the potential } \Delta_F \text{ is locally H\"older}   \right\}.
\end{equation}
\begin{rem} \label{bounded}
Note that if $F \in \mathcal{P}$ then there exists sequences $(s_n)_n$ and $(S_n)_n$ such that $s_n \leq \Delta_F|C_n \leq S_n$.
\end{rem}
Following the strategy developed by Barreira and Iommi in \cite{BI} we define the pressure.

\begin{defi} \label{def-pre}
Let $F \in \mathcal{P}$, the \emph{pressure} of $F$ with respect to the semi-flow $(Y, \tau)$ is defined by
\begin{equation}
 P_{\Phi}(F)= \inf \{t \in \R : P_{\sigma}(\Delta_F  -t \tau)  \leq 0  \}.
 \end{equation}
We assume the convention that $ P_{\Phi}(F)=\infty$ when the infimum is taken over the empty set.
\end{defi}

Of particular interest is the case in which the potential $F$ is the null potential. Indeed, in that case the pressure is equal to the entropy of the flow $P_{\Phi}(0)=h(\Phi)$. This case was studied by Polyakov \cite{p} and by Gurevich and Katok \cite{gk}. They obtained estimates for the entropy of the flow. In \cite{gk} the following bound was obtained
\begin{equation}
 0.7771 \leq h(\Phi) \leq 0.8161.
\end{equation}
Note that the entropy of the geodesic flow on the modular surface has entropy equal to one.

Let us make a few comments on Definition \ref{def-pre}. First note that there are potentials for which the equation $ P_{\sigma}(\Delta_F  -t \tau) =0$ does not have a root.

\begin{eje} \label{nroot}
Let $\rho:\Sigma_A \to \R$ be the locally constant potential defined by $\rho|C_n= \log (\log n)^{-2}$. In virtue of  Remark \ref{exten} there exists a potential $F \in \mathcal{P}$ such that $\Delta_F= \rho$. Let us extend $\rho$ to the full-shift on $\N$ defining it to be equal to zero on the cylinders where it remains to be defined. By Remark \ref{formula} we can bound the pressure by
\[P_{\sigma}( \Delta_F - t\tau)= P_{\sigma}( \log(\log n)^{-2}- t\log n^{-2}) < \log \left(\sum_{n=1}^{\infty} \frac{1}{n^{-2t}(\log n)^2} \right).\]
Therefore $P_{\Phi}(F)=1/2$ and $P_{\sigma}( \Delta_F - P_{\Phi}(F)\tau)<0.$
\end{eje}

When the equation defining the pressure has root, regularity properties of the pressure can be obtained.

\begin{teo}[Regularity]\label{reg}
Let $F \in \mathcal{P}$ with the property that there exists $\epsilon >0$ such that if $q \in (1-\epsilon, 1+ \epsilon)$ then the equation
\[ P_{\sigma}(\Delta_{qF}  -t \tau)  = 0  \]
has a unique root. Then, then the function
$q\mapsto P_\Phi(qF)$ is real analytic in the interval $(1-\epsilon, 1+ \epsilon)$.
\end{teo}

\begin{proof}
Since $\Sigma_A$ satisfies the BIP property and if $q \in (1-\epsilon, 1+ \epsilon)$ then $\Delta_{qF}$
is locally H\"older, the function $t\mapsto
P_\sigma(\Delta_{qF}-t\tau)$, when finite, is real analytic (see subsection \ref{cms} or
 \cite[Corollary 4]{sa3}). The result
now follows from the implicit function theorem: $P_\sigma(\Delta_{qF}
-P_\Phi(qF)\tau)=0$, and in order to verify the nondegeneracy
condition note that (see \cite[Chapter 4]{PU})
\[
\frac{\partial}{\partial t} P_\sigma(\Delta_{qF} -t\tau) \Big|_{t=s} =
-\int_\Sigma \tau \, d \mu < 0,
\]
where $\mu$ denotes the equilibrium (Gibbs) measure of $\Delta_{qF}-s\tau$ (the existence of such measure was proved in \cite[Theorem 1]{sa3}, see subsection \ref{cms}).
\end{proof}

The next two results are direct consequence of more general statements obtained in \cite[Theorem 1 and Theorem 2]{BI}, the proofs are the same.
Both are fundamental results in thermodynamic formalism. In particular they show that the notion of pressure we have introduced is a \emph{correct} definition.

\begin{teo}[Approximation property]\label{apro}
 If $F \in \mathcal{P}$  then
\[
P_\Phi(F)= \sup \left\{ P_{\Phi|K}(F) : K\subset Y \text{ compact
and $\Phi$-invariant} \right\},
\]
where $ P_{\Phi|K}(\cdot)$ is the pressure on compact spaces (see \cite[Chapter 9]{wa}).
 \end{teo}

\begin{teo}[Variational principle] \label{var1}
 If $F \in \mathcal{P}$  then
\begin{equation}\label{*bb71}
P_\Phi(F) = \sup \left\{ h_{\mu}(\Phi) +\int_Y F \,d \mu : \mu\in
\M_\Phi \text{ and } -\int_Y F \, d\mu <\infty \right\}.
\end{equation}
\end{teo}

A measure $\mu \in \M_{\Phi}$ such that $P_\Phi(F) = h_{\mu}(\Phi) +\int_Y F \,d \mu $ is called \emph{equilibrium measure} for $F$.

\begin{teo}[Equilibrium measures]\label{faa}
Let $F \in \mathcal{P}$ the following properties are equivalent:
\begin{enumerate}
\item
there is an equilibrium measure $\mu_F \in \M_\Phi$ for $F$;
\item $P_\sigma(\Delta_F -P_{\Phi}(F)\tau)=0$ and there
is an equilibrium measure $\nu_F\in \M_\sigma( \tau)$ for $\Delta_F -P_{\Phi}(F)\tau$.
\end{enumerate}
\end{teo}

\begin{proof}
We will assume that $P_{\Phi}(F) < \infty$, otherwise there is no equilibrium measure.
Recall that  from Definition \ref{presion} of topological pressure we have  that $P_\sigma(\Delta_F -P_{\Phi}(F)\tau)\le0$.

Let us first consider the case in which  $P_\sigma(\Delta_F -P_{\Phi}(F)\tau)<0$. Given $\nu\in \M_\sigma(\tau)$, by the variational principle for the pressure $P_{\sigma} ( \cdot)$ (see subsection \ref{cms}  or \cite[Theorem 3]{sa1}) we have
\[
h_\nu(\sigma) +\int_\Sigma \Delta_F \, d \nu -P_\Phi(F)
\int_\Sigma\tau \, d \nu <0.
\]
Since the space $\M_\sigma(\tau)$ can be identified with $\M_\Phi$, using
Abramov's formula (see subsection \ref{abramov}) we obtain that for every $\mu\in \M_\Phi$,
\[
h_{\mu}(\Phi) +\int_Y F \, d \mu < P_\Phi(F).
\]
Therefore, there are no equilibrium measures in this case.

Assume now that $P_\sigma(\Delta_F -P_{\Phi}(F)\tau)=0$, and let $\nu_F\in
\M_\sigma(\tau)$ be an equilibrium measure for the potential $\Delta_F -P_{\Phi}(F)\tau$. Then
\[
P_\sigma(\Delta_F -P_{\Phi}(F)\tau)= h_{\nu_F}(\sigma) +\int_\Sigma \left(\Delta_F -P_{\Phi}(F)\tau\right) \,d \nu_F = 0.
\]
Set $\mu_F= R(\nu_F)$. Since $\nu_F\in \M_\sigma(\tau)$ we have
$\int_\Sigma\tau \, d \nu_F <\infty$, and thus
\[
P_\Phi(F) = \frac{h_{\nu_F}(\sigma)}{\int_\Sigma\tau \, d \nu_F} +
\frac{\int_\Sigma \Delta_F\, d \nu_F}{\int_\Sigma\tau \, d \nu_F}
=h_{\mu_F}(\Phi) +\int_Y F \, d \mu_F.
\]
This shows that $\mu_F$ is an equilibrium measure for $F$. On the
other hand, if we start with an equilibrium measure $\mu_F$ for $F$,
then
\[
P_\Phi(F)= h_{\mu_F}(\Phi) +\int_Y F \, d \mu_F.
\]
The measure $\mu_F$ is obtained from a product measure $\nu_F\times
m$ for some $\nu_F\in \M_\sigma(\tau)$. Therefore, using Abramov's
formula,
\[
0 =P_\sigma(u_F)\ge h_{\nu_F}(\sigma) + \int_\Sigma \left( \Delta_F -P_{\Phi}(F)\tau \right) \, d
\nu_F=0.
\]
In particular, $\nu_F$ is an equilibrium measure for $u_F$. This
completes the proof.
\end{proof}

\begin{rem}
Note that every potential $\Delta_F-P_{\Phi}(F)\tau$ has a Gibbs measure $\mu \in \M_{\sigma}$ (see \cite[Theorem 1]{sa3}). Nevertheless, it is possible that this measure is not an equilibrium measure. Indeed,  if $h_{\sigma}(\mu)=\infty$ and $\int (\Delta_F-P_{\Phi}(F)\tau) d \mu = -\infty$ then $\mu$ is not an equilibrium measure.
\end{rem}

The following are examples in which  the conditions in part $(2)$ of Theorem \ref{faa}  fails.

\begin{eje}[No equilibrium measure 1]
Let $F:Y \to \R$ be the potential defined in Example \ref{nroot} then
\[ P_\sigma(\Delta_F -P_{\Phi}(F)\tau) <0. \]
Therefore, the potential $F$ does not have an equilibrium measure.
\end{eje}

\begin{eje}[No equilibrium measure 2]
Let $F:Y \to \R$ be a potential  such that $\Delta_F(x)=-\log(n(\log n)^2)$. Recall that the existence of such potential is guaranteed by  Remark \ref{exten}. By Remark \ref{formula}
we have that
\begin{equation*}
P((\Delta_F -t \tau)|\Sigma_F)= \log \sum_{n=6}^{\infty} \frac{1}{n^{1+2t}(\log n)^2}.
\end{equation*}
Therefore, we have
\[P((\Delta_F -P_{\Phi}(F) \tau)=0.\]
Moreover $P_{\Phi}(F) < 1$, indeed note that for $t=0$ the pressure function on the full-shift on $\N$, denoted by $\Sigma_\N$, is equal to zero. The result follows since $\Sigma_A \subset \Sigma_\N$.

The Gibbs measure $\mu \in \M_{\sigma}$ corresponding to the potential
$ \Delta_F -P_{\Phi}(F) \tau$, satisfies the following relation for a constant $C>0$ and for every natural number $n \geq 6$,
\[ \frac{1}{C}\leq \frac{\mu(C_n)}{n^{2P_{\Phi}(F)+1}(\log n)^2}\leq C. \]
Note that
 \begin{equation*}
\int \tau \ d \mu_{\phi} \geq \sum_{n=6}^{\infty} n^2 \mu(C_n) \geq \frac{1}{C}  \sum_{n=6}^{\infty} \frac{n^2}{n^{2P_{\Phi}(F)+1}(\log n)^2} = \frac{1}{C}  \sum_{n=6}^{\infty} \frac{1}{n^{2P_{\Phi}(F)-1}(\log n)^2} \end{equation*}
Since $2P_{\Phi}(F)-1<1$ we have that $\int \tau \ d \mu_{\phi}= \infty$.
\end{eje}

In the next Theorem we establish conditions under which a potential has an equilibrium measure. This condition is a \emph{small oscillation} type of condition.
\begin{teo}
Let $F \in \mathcal{P}$ be a bounded potential with $P_{\Phi}(F) < \infty$. If
\begin{equation} \label{va}
\sup F - \inf F < h_{top}(\Phi) - \frac{1}{2},
\end{equation}
then $F$ has an equilibrium measure $\mu_F    \in \M_{\Phi}$.
 \end{teo}

\begin{proof}
Let us first show that $P_\sigma(\Delta_F -P_{\Phi}(g)\tau)=0$.
Denote by $s= \inf F$ and by $S=\sup F$. We have that $s\tau\le\Delta_F\le S\tau$. Moreover,
\[
P_\sigma((s-t)\tau) \le P_\sigma(\Delta_F -t\tau) \le
P_\sigma((S-t)\tau).
\]
Let $t^* \in \left(1/2 +S , h_{top}(\Phi)+s \right)$ (by equation \eqref{va} this interval is non degenerate). We have that
\[
0< P_\sigma((s-t^*)\tau) \le P_\sigma(\Delta_F -t^*\tau) \le
P_\sigma((S-t^*)\tau) < \infty.
\]
Since $P_{\Phi}(F) < \infty$ and by the continuity of the function $t \to  P_\sigma(\Delta_F -t\tau)$ we have that $P_\sigma(\Delta_F -P_{\Phi}(g)\tau)=0$.

It remains to show that the potential $\Delta_F -P_{\Phi}(F)\tau $ has an equilibrium measure. Recall that there exists a Gibbs measure $\mu$ associated to this potential. In order to show that this measure is an equilibrium measure it suffices to prove that
\begin{equation} \label{der}
\int \left(  \Delta_F -P_{\Phi}(F)\tau \right) \ d \mu < \infty.
\end{equation}
But note that \cite[Chapter 4]{PU}
\[ \frac{d}{dt}  P_\sigma(\Delta_F -t\tau) \Big{|}_{t=P_{\Phi}(F)}=  \int \left(  \Delta_F -P_{\Phi}(F)\tau \right) \ d \mu. \]
Note that we have proved that there exists an interval of the form $[P_{\phi}(F)-\epsilon, P_{\phi}(F)+\epsilon]$ where the function $t \to P_\sigma(\Delta_F -t\tau)$ is finite.
The result now follows, because when finite the function $t \to P_\sigma(\Delta_F -t\tau)$ is real analytic.
\end{proof}

It is interesting to remark that in a  wide range of different contexts similar small oscillation conditions on the potential have been imposed in order to prove existence and uniqueness of equilibrium measures. For instance, Hofbauer \cite{h} originally in a symbolic setting  and later Hofbauer and Keller   \cite{hk}
in the context of the angle doubling map on the circle, gave examples which shows that  this type of conditions was essential in their setting in order to have quasi-compactness of the transfer operator, and hence good equilibrium measures. Later, Denker and Urba\'nski \cite{du}, for rational maps, used similar conditions to prove uniqueness of equilibrium measures. Oliveira \cite{o} proves the existence of equilibrium measures for potentials satisfying similar conditions for  non-uniformly expanding maps.  Recently, Bruin and Todd \cite{bt} in the context of multimodal maps also made use of a similar condition to prove uniqueness of equilibrium measures.


\begin{thebibliography}{99}

\bibitem[A]{a}   L. M. Abramov. \emph{On the entropy of a flow.} Dokl. Akad. Nauk SSSR 128 (1959) 873--875.

\bibitem[AK]{ak}  W. Ambrose  and  S. Kakutani. \emph{Structure and continuity of measurable flows.} Duke Math. J. 9, (1942). 25--42.

\bibitem[BI]{BI} L. Barreira  and G. Iommi. \emph{Suspension flows over countable Markov shifts} J. Stat. Phys. 124 (2006), no. 1, 207--230.


\bibitem[BRW]{brw}
L. Barreira, L. Radu and  C. Wolf. \emph{Dimension of measures for
suspension flows}, Dyn. Syst. \textbf{19} (2004), 89--107.




\bibitem[B1]{bo1}
R. Bowen \emph{Symbolic dynamics for hyperbolic flows}, Amer. J.
Math. \textbf{95} (1973), 429--460.


\bibitem[BR]{br} R. Bowen and  D. Ruelle. \emph{The ergodic theory of Axiom A flows.} Invent. Math. 29 (1975), no. 3, 181--202

\bibitem[BT]{bt} H. Bruin and  M. Todd. \emph{Equilibrium states for interval maps: potentials with $\sup\phi-\inf\phi<h_{\rm top}(f)$.} Comm. Math. Phys. 283 (2008), no. 3, 579--611.
\bibitem[DU]{du}  M. Denker  and M. Urba\'nski. \emph{Ergodic theory of equilibrium states for rational maps.} Nonlinearity 4 (1991), no. 1, 103--134.

\bibitem[GK]{gk} B. Gurevich  and  S. Katok. \emph{Arithmetic coding and entropy for the positive geodesic flow on the modular surface. Dedicated to the memory of I. G. Petrovskii on the occasion of his 100th anniversary.} Mosc. Math. J. 1 (2001), no. 4, 569--582, 645.

\bibitem[H]{h} F. Hofbauer. \emph{Examples for the nonuniqueness of the equilibrium state.} Trans. Amer. Math. Soc. 228 (1977), no. 223--241.

\bibitem[HK]{hk}  F. Hofbauer  and G. Keller. \emph{Equilibrium states for piecewise monotonic transformations.} Ergodic Theory Dynam. Systems 2 (1982), no. 1, 23--43.

\bibitem[IJ]{ij} G. Iommi and T. Jordan \emph{Phase transitions for suspension flows}
arXiv:1202.0849

\bibitem[K1]{ka} S. Katok. \emph{Fuchsian groups.} Chicago Lectures in Mathematics. University of Chicago Press, Chicago, IL, (1992). x+175 pp

\bibitem[K2]{k}  S. Katok. \emph{Coding of closed geodesics after Gauss and Morse.} Geom. Dedicata 63 (1996), no. 2, 123--145.

\bibitem[KU]{ku}  S. Katok and  I. Ugarcovici. \emph{Symbolic dynamics for the modular surface and beyond.} Bull. Amer. Math. Soc. (N.S.) 44 (2007), no. 1, 87--132

\bibitem[MU]{mu}  R.D. Mauldin  and M. Urba\'nski. \emph{Dimensions and measures in infinite iterated function systems.} Proc. London Math. Soc. (3) 73 (1996), no. 1, 105--154.

\bibitem[O]{o} K. Oliveira. \emph{Equilibrium states for non-uniformly expanding maps.} Ergodic Theory Dynam. Systems 23 (2003), no. 6, 1891--1905.

\bibitem[PP]{pp}  W. Parry and  M. Pollicott. \emph{Zeta functions and the periodic orbit structure of hyperbolic dynamics.} AstŽrisque No. 187-188 (1990), 268 pp.


\bibitem[Pe]{pe}
Y. Pesin.  \emph{Dimension Theory in Dynamical Systems: Contemporary
Views and Applications}, Chicago Lectures in Mathematics, Chicago
University Press, (1997).

\bibitem[Po]{p}  A.B. Polyakov \emph{On a measure with maximal entropy for a special flow over a local perturbation of a countable topological Bernoulli scheme.} (Russian) Mat. Sb. 192 (2001), no. 7, 73--96; translation in Sb. Math. 192 (2001), no. 7-8, 1001--1024


\bibitem[PU]{PU} F. Przytycki  and  M. Urba\'nski.
\emph{Fractals in the Plane, Ergodic Theory Methods,}
to appear in Cambridge University Press. Available at http://www.math.unt.edu/~urbanski/book1.html.


\bibitem[R]{ra}
M. Ratner. \emph{Markov partitions for Anosov flows on
$n$-dimensional manifolds}, Israel J. Math. \textbf{15} (1973),
92--114.


\bibitem[S1]{sa1} O. Sarig. \emph{ Thermodynamic formalism for
    countable Markov shifts,}  Ergodic Theory Dynam. Systems {\bf
    19}  (1999) 1565--1593.

\bibitem[S2]{sa2}  O. Sarig. \emph{Phase transitions for countable Markov shifts,}  Comm. Math. Phys.  \textbf{217}  (2001) 555--577.

\bibitem[S3]{sa3}  O. Sarig. \emph{Existence of Gibbs measures for countable Markov shifts,} Proc. Amer. Math. Soc. \textbf{131} (2003) 1751--1758


\bibitem[Sav]{sav}
S. Savchenko. \emph{Special flows constructed from countable
topological Markov chains}, Funct. Anal. Appl. \textbf{32} (1998),
32--41.



\bibitem[W]{wa}
P. Walters. \emph{An Introduction to Ergodic Theory}, Graduate Texts
in Mathematics 79, Springer, (1981).


\end{thebibliography}
\end{document}